\documentclass[12pt]{article}
\usepackage{amsmath,amssymb,amsthm, amsfonts}
\usepackage{hyperref}
\usepackage{graphicx}
\usepackage{url}
\usepackage{dsfont} 
\usepackage{color}
\definecolor{red}{rgb}{1,0,0}

\definecolor{blue}{rgb}{0,0,1}

\definecolor{green}{rgb}{0,.6,0}

\usepackage{float}
\usepackage{lineno}

\usepackage{tikz}
\usetikzlibrary{shapes.geometric}
\usetikzlibrary{shapes.multipart}
\usetikzlibrary{arrows}

\usetikzlibrary{calc}
\usetikzlibrary{arrows.meta}
\usetikzlibrary{arrows}
\usetikzlibrary{patterns}
\usetikzlibrary{shapes,snakes}
\usetikzlibrary{decorations.markings}
\tikzset{vtx/.style={inner sep=1.7pt, outer sep=0pt, circle,fill,draw}, 
vtx red/.style={inner sep=2.7pt, outer sep=0pt, circle, fill=red,draw}, 
vtx blue/.style={inner sep=2.7pt, outer sep=0pt, circle, fill=blue,draw}, 
vtx green/.style={inner sep=2.7pt, outer sep=0pt, circle, fill=black!30!green,draw}, 
vtx gray/.style={inner sep=2.7pt, outer sep=0pt, circle, fill=gray!30!white,draw}, 
vtx3/.style={inner sep=1.2pt, outer sep=0pt, regular polygon, regular polygon sides=3, draw},
vtx4/.style={inner sep=1.7pt, outer sep=0pt, regular polygon, regular polygon sides=4, draw},
vtx5/.style={inner sep=1.7pt, outer sep=0pt, regular polygon, regular polygon sides=5, draw},
vtx_box/.style={inner sep=2pt, outer sep=0pt, rectangle, draw, fill=white},
blueish/.style={color=blue},
reddish/.style={color=red},
greenish/.style={color=black!30!green},
edge_blueish/.style={color=blue,line width=2pt},
edge_reddish/.style={color=red,line width=2pt,dashed},
edge_greenish/.style={color=black!30!green,line width=2pt, densely dotted},
edge_orange/.style={color=orange,line width=2pt},
directed edge/.style={decoration={markings, mark=at position 0.5 with {\arrow{Latex}}},  postaction={decorate}},
} 

\newtheorem{thm}{Theorem}[section]

\newtheorem{lem}[thm]{Lemma}
\newtheorem{prop}[thm]{Proposition}

\newtheorem{obs}[thm]{Observation}

\theoremstyle{definition}

\theoremstyle{definition}
\newtheorem{defn}[thm]{Definition}

\theoremstyle{definition}

\newcommand{\thperc}{\operatorname{th}_p^{(r)}}
\newcommand{\ptime}{\operatorname{perc}_r(G;B)}

\newcommand{\R}{\mathbb{R}}

\newcommand{\N}{\mathbb{N}}

\newcommand{\Z}{\operatorname{Z}}
\newcommand{\X}{\operatorname{X}}

\newcommand{\bit}{\begin{itemize}}
\newcommand{\eit}{\end{itemize}}
\newcommand{\ben}{\begin{enumerate}}
\newcommand{\een}{\end{enumerate}}
\newcommand{\beq}{\begin{equation}}
\newcommand{\eeq}{\end{equation}}
\newcommand{\bea}{\begin{eqnarray}} 
\newcommand{\eea}{\end{eqnarray}}
\newcommand{\bpf}{\begin{proof}}
\newcommand{\epf}{\end{proof}\ms}
\newcommand{\bmt}{\begin{bmatrix}}
\newcommand{\emt}{\end{bmatrix}}
\newcommand{\ms}{\medskip}

\newcommand{\noi}{\noindent}

\newcommand{\beqs}{\begin{equation*}} 
\newcommand{\eeqs}{\end{equation*}}
\newcommand{\beas}{\begin{eqnarray*}}
\newcommand{\eeas}{\end{eqnarray*}}

\newcommand{\up}[1]{^{(#1)}}
\newcommand{\upc}[1]{^{[#1]}}

\newcommand{\calf}{\mathcal{F}}

\newcommand{\zf}{\operatorname{\lfloor \operatorname{Z} \rfloor}}

\newcommand{\zp}{\operatorname{Z}_{+}}

\newcommand{\pt}{\operatorname{pt}}

\newcommand{\ptx}{\operatorname{pt}_{\X}}

\newcommand{\throt}{\operatorname{th}}

\newcommand{\thx}{\operatorname{th_{\X}}}
\newcommand{\thstar}{\throt^*}
\newcommand{\thwx}{\throt^{\omega}_{\X}}

\newcommand{\fs}{\rightarrow}

\title{Generalizing forbidden induced subgraph characterizations of high throttling numbers}
\author{Joshua Carlson\thanks{Dept.~of Mathematics and Computer Science, Drake University, Des Moines, IA, USA (jc31@williams.edu)} \and 
J\"urgen Kritschgau\thanks{Dept.~of Mathematics, Carnegie Mellon University, Pittsburgh, PA, USA (jkritsch@andrew.cmu.edu)} }

\date{\today}

\begin{document}
\maketitle

\begin{abstract} 
Zero forcing is a process that models the spread of information throughout a graph as white vertices are forced to turn blue using a color change rule. The idea of throttling, introduced in 2013 by Butler and Young, is to optimize the trade-off between the number of initial blue vertices and the time taken to force all vertices to become blue. The original throttling number of a graph minimizes the sum of these two quantities and the product throttling number minimizes their product. In addition, weighted throttling changes the weights given to these two quantities when minimizing their sum. Since its introduction, throttling has expanded to include many variants of zero forcing. This motivates the study of zero forcing and throttling using abstract color change rules. Recently, it has been shown that the graphs with high (sum) throttling numbers are characterized by a finite family of forbidden induced subgraphs. In this paper, we extend that result to throttling, product throttling, and weighted throttling using abstract color change rules. To this end, we define some important families of color change rules and explore their properties.
\end{abstract}

\noi {\bf Keywords} Zero forcing, throttling, forbidden subgraphs, color change rule

\noi{\bf AMS subject classification} 05C57, 05C15, 05C50
\section{Introduction}

Zero forcing is a combinatorial game played on graphs in which a color change rule is used to change the color of vertices from white to blue.  The \emph{standard color change rule}, denoted $\Z$, states that if a blue vertex $v$ has a unique white neighbor $w$, then $v$ can force $w$ to become blue. Starting with an initial subset of vertices $B \subseteq V(G)$ colored blue and $V(G) \setminus B$ colored white, the goal of the game is to repeatedly apply the color change rule and eventually force every vertex in $V(G)$ to become blue. If this goal is achievable using the standard color change rule, then the initial subset $B$ of blue vertices is called a \emph{(standard) zero forcing set} of $G$. The \emph{(standard) zero forcing number} of a graph $G$, denoted $\Z(G)$, is the size of a minimum standard zero forcing set of $G$.

The standard zero forcing number was introduced in \cite{AIM} as an upper bound for the maximum nullity of a family of symmetric matrices that correspond to a given graph. Since its origin, many variants of zero forcing, obtained by altering the standard color change rule, have been studied (see \cite{Parameters}). One variant, called \emph{positive semidefinite (PSD) zero forcing}, allows forcing to occur in multiple components of a graph. Suppose $G$ is a graph and $B \subseteq V(G)$ is the set blue vertices in $V(G)$. Let $W_1, W_2, \ldots, W_k$ be the sets of white vertices in the components of $G-B$ respectively. The \emph{PSD color change rule}, denoted $\zp$, states that if $v$ is a blue vertex and $w$ is the unique white neighbor of $v$ in the graph $G[B \cup W_i]$ for some $1 \leq i \leq k$, then $v$ can force $w$ to become blue. Note that the PSD color change rule is simply the standard color change rule applied within each component of $G-B$. PSD forcing sets and the PSD forcing number, denoted $\zp(G)$, are defined analogously to their standard counterparts.

Due to the numerous variants of zero forcing, attempts have been made to unify these parameters with abstract definitions (see \cite{Parameters, JCThrot, CK20}). All zero forcing parameters stem from a color change rule that specifies the conditions under which a vertex $v$ can force another vertex $w$ to become blue. For a given color change rule $\X$, a valid force can be denoted as $v \overset{\X}{\rightarrow} w$. In a graph $G$, suppose we start with $B \subseteq V(G)$ colored blue and $V(G) \setminus B$ colored white and we apply an arbitrary color change rule $\X$ until no more forces are possible. The set of blue vertices in $G$ that results from this process is called an \emph{$\X$ final coloring} of $B$. An \emph{$\X$ forcing set} of $G$ is a subset of $V(G)$ that has $V(G)$ as an $\X$ final coloring and the \emph{$\X$ forcing number} for $G$, denoted $\X(G)$, is the size of a minimum $\X$ forcing set of $G$. In general, there can be many distinct $\X$ final colorings of a given subset $B \subseteq V(G)$ (see \cite{Parameters, JCThrot}). For $\Z$ and $\zp$, the final coloring of a given subset $B \subseteq V(G)$ is unique and is sometimes called the \emph{closure} of $B$. 

There are various ways to keep track of the forces that occur during a zero forcing process. If $B \subseteq V(G)$ is the initial set of blue vertices, the ordered list of forces performed to obtain an $\X$ final coloring of $B$ is called a \emph{chronological list of $\X$ forces of $B$}. Also, the set of forces that appear in a given chronological list is called a \emph{set of $\X$ forces of $B$}. For a set of $\X$ forces, $\calf$, an \emph{$\X$ forcing chain of $\calf$} is a list of vertices $v_1, v_2, \ldots, v_k \in V(G)$, such that the force $(v_i \overset{\X}{\rightarrow} v_{i+1}) \in \calf$ for each $1 \leq i \leq k-1$. Note that this definition of forcing chain does not require the list of vertices in the chain to be maximal. This diverges from some previous literature, but is necessary for our investigation.

In addition to its connections to linear algebra, zero forcing is also studied for its combinatorial properties. There are a variety of parameters that measure the time taken during a zero forcing process. Suppose $\calf$ is a set of $\X$ forces of a subset $B \subseteq V(G)$. Define $\calf \up{0} = B$ and for each integer $t>0$, define $\calf \up{t}$ as follows. First, color $\bigcup_{i=0}^{t-1} \calf \up{i}$ blue and color $V(G) \setminus \bigcup_{i=0}^{t-1} \calf \up{i}$ white. Given this coloring, $\calf \up{t}$ is the set of white vertices $w$ for which there exists a blue vertex $b$ such that $(b \rightarrow w)$ is a valid $\X$ force in $\calf$.
For each integer $t \geq 0$, let $\calf\upc{t} = \bigcup_{i=0}^t \calf\up{i}$.
If $\calf\up i$ is uniquely determined by $B$ independent of $\calf$, then we write $B\up i=\calf\up i$ and $B\upc{i} = \calf \upc{i}$.
The \emph{$\X$ propagation time} of $\calf$, denoted $\ptx(G; \calf)$, is the smallest nonnegative integer $q$ such that $\calf\upc{q} = V(G)$.  Note that if $B$ is not an $\X$ forcing set of $G$, then $\ptx(G; \calf) = \infty$ for any set of $X$ forces $\calf$ of $B$. Since there are sometimes many distinct sets of  $\X$ forces of a given subset $B \subseteq V(G)$, the \emph{$\X$ propagation time of $B$} is defined as $\ptx(G;B) = \min \{\ptx(G; \calf) \ | \ \calf \text{ is a set of $\X$ forces of $B$}\}.$ 

Informally, $\calf \upc{t}$ is the set of vertices in $V(G)$ that are blue at \emph{time} $t$ and for each $t > 0$, $\calf\up{t}$ is the set of vertices that turn blue during \emph{time step} $t$. In addition, let $U_{\mathcal F}\up{t}$ denote the set of vertices that perform the forces during each positive time step $t$ and let $U_{\mathcal F}\up{0} = \emptyset$. Analogously, for each integer $t \geq 0$, the set $U_{\mathcal F}\upc{t} = \bigcup_{i=0}^t U_{\mathcal F}\up{i}$ is the set of vertices that have performed a force by time $t$.

It is clear that the the size of the initial set $B \subseteq V(G)$ of blue vertices and the propagation time of $B$ are both important throughout a zero forcing process. As such, there are a variety of parameters that combine both of these quantities. For a subset $B \subseteq V(G)$, the \emph{$\X$ throttling number of $B$} is $\thx(G;B) = |B| + \ptx(G;B)$ and the \emph{$\X$ throttling number} of $G$ is $\thx(G) = \min\{\thx(G;B) \ | \ B \subseteq V(G)\}$. The concept of throttling for zero forcing was first introduced by Butler and Young in \cite{BY13Throt}. In recent years, throttling has become a significant area of research which has expanded to include many variants of zero forcing (see \cite{powerdomthrot, BY13Throt, JCThrot, PSD, skew}) and some variants of the game of cops and robbers (see \cite{CopThrot2, CRthrottle, damagethrot}). 

Sometimes, we may not want to give $|B|$ and $\ptx(G;B)$ equal weight when minimizing their sum. If $\omega$ is a nonnegative real number, the \emph{weighted $\X$ throttling number of $G$} is $\thwx(G) = \min \{|B| + \omega \cdot \ptx(G;B) \ | \ B \subseteq V(G)\}.$ We can also minimize a product instead of a sum. The \emph{(no initial cost) $\X$ product throttling number of $G$} is \[\thstar_{\X}(G) = \min\{|B| \cdot \ptx(G;B) \ | \ B \subsetneq V(G)\}.\] In each definition that involves an abstract color change rule $\X$, the $\X$ can be dropped if the exact rule is clear from context. 

In \cite{CK20}, the authors show that the problem of determining graphs with high standard or PSD throttling numbers is a forbidden subgraph problem.
\begin{thm}{\cite[Theorem 4.7]{CK20}}\label{thrm:CK20}
Suppose that $\X$ is either the standard or PSD color change rule and $k$ is a nonnegative integer. The set of graphs $G$ such that $\thx(G) \geq |V(G)| - k$ and $|V(G)| \geq k$ is characterized by a finite family of forbidden induced subgraphs.
\end{thm}

Given that abstract color change rules can behave very differently and there are multiple types of throttling numbers, it is natural to ask the following question. To what extent can Theorem \ref{thrm:CK20} be generalized? In order to provide an answer to this question, we begin by identifying some convenient properties held by many zero forcing color change rules. 

\begin{defn}\label{def:wellBehaved}
A forcing color change rule $\X$ is \emph{well behaved} if for any graph $G$, any $\X$ forcing set $B \subseteq V(G)$ with any  set of $\X$ forces $\calf$ such that $\ptx(G;\calf)<\infty$,
\begin{enumerate}
    \item if $B$ is an $\X$ forcing set of $G$, then any super set of $B$ is an $\X$ forcing set of $G$,
    \item $|U_{\mathcal F}\up i|\leq |\calf\up i|$ for all $0\leq i\leq \ptx(G;\calf)$,
    \item $\calf ^{[i]}\setminus U_{\mathcal F}^{[i]}$ is an $\X$ forcing set of $G-U_{\mathcal F}^{[i]}$ for all $0\leq i\leq \ptx(G;\calf )$, and
    \item $U_{\mathcal F}^{[t]}$ is an $\X$ forcing set for $G-(B\setminus U_{\mathcal F}^{[t]})$ for $t=\ptx(G,\calf)$.
\end{enumerate}
We say a color change rule $\X$ is $\emph{nearly well behaved}$ if $\X$ only satisfies properties $1$ and $2$.  
\end{defn}

The conditions in Definition \ref{def:wellBehaved} seem restrictive at first, but they are in fact fairly natural for many variants of zero forcing. 
The first condition allows us to add vertices to $\X$ forcing sets.
The second condition ensures that in each time step, the number of new blue vertices is at least the number of vertices that performed a force. 
This condition is  trivially satisfied by most zero forcing color change rules since they typically specify that a particular vertex $v$ forces a white vertex $w$. 
In other words, the color change rule provides a surjective map from $\calf \up i $ to $U_{\mathcal F}\up i$, relating their cardinalities. 
Interestingly, the definition of well behaved allows for multiple  vertices to cooperatively force a white vertex, as long as the cardinality condition is met.
The third condition says that vertices that have already performed a force are not necessary for future forces and the fourth condition says that blue vertices in $B$ that never perform a force are not needed at all.

In Section \ref{sec:weighted}, we extend Theorem \ref{thrm:CK20} to weighted throttling for well behaved color change rules. Then, we introduce chain independent color change rules in order to investigate product throttling in Section \ref{sec:productThrottling}. In Section \ref{sec:colocal}, we make connections to bootstrap percolation by introducing color change rules whose forcing conditions depend on the neighborhood of the vertex being forced (rather than the vertex doing the forcing). All graphs in this paper are simple, finite, and undirected. Furthermore, we follow most of the graph theoretic notation found in \cite{Diestel}.

\section{Weighted throttling}\label{sec:weighted}


In this section, we examine the weighted throttling number $\thwx(G)$ for a well behaved color change rule $\X$. Note that when $\omega = 1$, $\thwx(G)$ specializes to the classic throttling number. If $\omega \geq 0$ and $\X$ is a well behaved color change rule, the following proposition demonstrates that high values of $\thwx(G)$ can be characterized using forbidden subgraphs.

\begin{prop}\label{forbidden}
Let $k$ be a constant and $\X$ be a well behaved color change rule. The set of graphs $G$ such that  $\thwx\geq |V(G)|-k$  and $|V(G)|\geq k$ is characterized by a family of forbidden induced subgraphs. 
\end{prop}

\begin{proof}
 Suppose that $\thwx (G)<|V(G)|-k$ and let $H$ be any graph such that $G$ is an induced subgraph of $H$ with the injection $\phi:V(G)\to V(H)$.
 Let $B\subseteq V(G)$ be an  $\X$ forcing set that realizes $\throt(G; B) = \thwx (G) < V(G)-k$ and let $W = V(G) \setminus B$. 
 Then $B'=V(H)\setminus \phi(W)$ is an $\X$ forcing set of $H$.  
 This follows from the fact that if $v\rightarrow u$ is possible in $G$ given $B$, then $\phi(v)\rightarrow \phi(u)$ is possible in $H$ given $B'$ since $\X$ is well behaved.
 In particular, \[\thwx(H)\leq |B'|+\omega \cdot\ptx(H;B')=|V(H)\setminus \phi(V(G))|+|B|+\omega\cdot\ptx(G;B)<|V(H)|-k.\] Therefore, $B'$ is an $\X$ forcing set of $H$ that demonstrates that $\thwx(H)<V(H)-k$.
\end{proof}

Our goal is to show that the family of forbidden subgraphs in Propositon \ref{forbidden} is finite. This is achieved in \cite{CK20} for $\omega = 1$ and $\X \in \{\Z, \zp\}$ by making use of specific zero forcing sets called standard witnesses. The next definition formalizes this concept for weighted throttling.

\begin{defn}
We say an $\X$ forcing set $B\subseteq V(G)$ is a \emph{standard witness for $\thwx(G)< |V(G)|-k$}, if $|\calf \up i|-\omega>0$ for each time step $i$ and $|B| + \omega \cdot\ptx(G; B) < |V(G)|-k$ where $\calf$ is a set of forces such that $\ptx(G;\calf)=\ptx(G;B)$. 
\end{defn}

Before we prove the main result in this section, we need to establish some preliminary facts about standard witnesses and the weighted throttling number.

\begin{lem}\label{switexists}
Suppose $\X$ is a well behaved color change rule.
If $\thwx(G)< |V(G)|-k$, then there exists a standard witness $\thwx(G)< |V(G)|-k$.
\end{lem}

\begin{proof}
Let $B$ be an $\X$ forcing set such that $\thwx(G;B)< |V(G)|-k$. Let $\calf$ be a set of forces for $B$ that realizes $\ptx(G;B)$.
Let $I$ be the set of times $i$ such that $|\calf\up i|-\omega\leq 0$.
Then \[B'=B\cup \bigcup_{i\in I} \calf\up i\] is a standard witness for $\thwx(G)< |V(G)|-k$ since $\X$ is well behaved.
\end{proof}

\begin{lem}\label{savings}
Let $G$ be a graph, $\X$ be a color change rule, and $\omega$ be a non-negative real number. Then, $\thwx(G) < |V(G)|-k$ if and only if there exists an $\X$ forcing set $B \subseteq V(G)$ and set of forces $\calf$ such that $\ptx(G;B)=\ptx(G;\calf)$ and 
 \[\sum_{i=1}^{\ptx(G;B)} |\calf^{(i)}|-\omega  > k.\] 
\end{lem}

\begin{proof}
Let $B$ be an $\X$ forcing set of $G$ with \[\sum_{i=1}^{\ptx(G;B)} |\calf ^{(i)}|-\omega  > k.\] This implies that 
\begin{align*}
    |V(G)\setminus B|-\omega\cdot\ptx(G;B)&> k \\
    |V(G)|-|B|-\omega\cdot\ptx(G;B)&> k\\
    |V(G)|-k&> |B|+\omega\cdot\ptx(G;B)\\
    |V(G)|-k&> \thwx(G).
\end{align*}

To prove the converse, assume that $|V(G)|-k>\thwx(G)$ and let $B$ be an $\X$ forcing set that realizes this inequality. In particular, suppose that
\[|V(G)|-k> |B|+\omega\cdot\ptx(G;B).\]
This implies that 
\[|V(G)\setminus B|-\omega\cdot\ptx(G;B)> k.\]
 Since $B$ is an $\X$ forcing set, we can partition $V(G)\setminus B$ into $\calf ^{(i)}$ for $1\leq i \leq \ptx(G;B)$.
 Using this partition, we can count the elements in $V(G)\setminus B$ to obtain 
 \[\sum_{i=1}^{\ptx(G;B)} |\calf ^{(i)}|-\omega >k.\]
 This completes the proof.
\end{proof}

The following theorem extends Theorem \ref{thrm:CK20} to well behaved color change rules and arbitrary non-negative weights.

\begin{thm}\label{finite}
Let $\X$ be a well behaved color change rule. The set of graphs $G$ such that $\thwx(G)\geq |V(G)|-k$ and $|V(G)|\geq k$ is characterized by a finite family of forbidden induced subgraphs. 
\end{thm}

\begin{proof}
Notice that we can write $\omega=\omega_\N+\omega_\R$ where $\omega_\N\in \N$ and $0\leq \omega_\R<1$ with $\omega_\R\in \R$.
Let $k$ be a non-negative integer and $\mathcal G$ be the set of all graphs $G$ such that $\thwx (G)<|V(G)|-k$ and $|V(G)|\leq2(k+1)+\frac{2\omega(k+1)}{1-\omega_\R}$. 
We will prove the claim that if $\thwx(G)<|V(G)|-k$ and $|V(G)|\geq k$, then $G$ contains a graph in $\mathcal G$ as an induced subgraph. 
By Lemma \ref{savings}, there exists a zero forcing set $B$  and set of forces $\calf$ such that $\ptx(G;B)=\ptx(G;\calf)$ and
\[\sum_{i=1}^{\ptx^\omega(G;B)} |\calf^{(i)}|-\omega  > k.\]
Without loss of generality, assume that $B$ is a standard witness for $\thwx(G)<|V(G)|-k$.
Let $\X$ be the first time step at which $\sum_{i=1}^{r} |\calf^{(i)}|-\omega  >k.$ 
In fact, we can choose $\hat \calf ^{(r)}\subseteq \calf^{(r)}$ so that \[|\hat \calf^{(r)}|-\omega+\sum_{i=1}^{r-1} |\calf^{(i)}|-\omega \leq k+1.\] 
To avoid cumbersome notation, let $\hat \calf^{(i)}= \calf\up i$ for each $1\leq i \leq r-1$ so that \[\sum_{i=1}^{r}|\hat \calf\up{i}| - \omega \leq  k+1.\]
Since $B$ is a standard witness for $\thwx(G) < |V(G)| - k$, $r\leq \frac{k+1}{1-\omega_\R}$. 
Let $H=G[S]$ where \[S=\bigcup_{i=1}^r U_{\mathcal F}^{(i)}\cup \hat \calf^{(i)}.\]

First, we will show that $\thwx (H)< |V(H)|-k$. 
Then, we will show that $|V(H)|\leq 2(k+1)+\frac{2\omega(k+1)}{1-\omega_\R}$. 
This will prove that $H$ is in $\mathcal G$.

Let \[\hat B=\bigcup_{i=1}^r \left(U_{\mathcal F}^{(i)}\setminus \bigcup_{j=1}^{i-1}\hat  B^{(j)}\right).\] 
We will prove that $\hat B\up i$ is blue after time step $i$ by induction on $i$, assuming that $\hat B$ is the initial zero forcing set. 
As a base case, $\hat B$ is a set of blue vertices in $H$ after $0$ time steps by construction. 
We will assume that the sets $\hat \calf^{(j)}$ for $0\leq j\leq i-1$ are blue at the beginning of time step $i$.
This implies that $U_{\mathcal F}^{(i)}$ is blue at the beginning of time step $i$. 
Since $H$ is an induced subgraph of $G$ that contains $U_{\mathcal F}^{(i)}$ and $\hat \calf^{(i)}$, the set $U_{\mathcal F}^{(i)}$ can force $\hat \calf^{(i)}$ in $H$. 
Therefore, after time step $i$, the vertices in $\hat \calf^{(i)}$ are blue in $H$. 
Thus, $\hat B$ can force all of $H$ in at most $\X$ time steps. 
Now, 
\[\thwx(H)\leq |V(H)|-\sum_{i=1}^r |\hat \calf^{(i)}|-\omega< |V(H)|-k\]
by Lemma 4.4.

Notice that $|U_{\mathcal F}^{(i)}|\leq |\hat \calf^{(i)}|$ by the $\X$ color change rule (this is an equality for standard zero forcing, but can be an inequality for PSD zero forcing). 
Therefore,
\[ |S|\leq \sum_{i=1}^{r} |U_{\mathcal F}^{(i)}|+|\hat \calf^{(i)}|\leq 2\sum_{i=1}^{r}|\hat \calf^{(i)}|\leq 2(k+1+r\omega)\leq 2(k+1)+\frac{2\omega(k+1)}{1-\omega_\R}.\] 
Thus, $H=G[S]$ is a graph in $\mathcal G$. 
\end{proof}

Counting exactly how many graphs are forbidden  seems hard. 
The size of the largest forbidden graph serves as an intuitive proxy for the number of graphs that are forbidden. 
A detail hidden in the proof of Theorem \ref{finite} is that the largest forbidden graph in the characterizing family has order at most  $2(k+1)+\frac{2\omega(k+1)}{1-\omega_\R}$.
By contrast, the largest graph forbidden in the proof of Theorem \ref{thrm:CK20} (which is the unweighted analog of Theorem \ref{finite}) has order $4k+4$. 
This quantity is recovered exactly when $\omega =1$.

\section{Product throttling}\label{sec:productThrottling}

In this section, we investigate no initial cost product throttling for abstract color change rules. 
An overview of the current literature on product throttling is given in \cite{survey} where the authors set up the following notation.
Let $G$ be a graph, $\X$ be an abstract color change rule, and  $k$ be a non-negative integer. Then, $\ptx(G,k)$ is the minimum value of $\ptx(G;B)$ where $B$ ranges over all $\X$ forcing sets of $G$ of size $k$.
Furthermore, for a non-negative integer $p$, $k_{\X}(G,p)$ is the minimum cardinality of an $\X$ forcing set $B$ such that $\ptx(G;B) = p$.
The following theorem concerns no initial cost product throttling for standard zero forcing.

\begin{thm}\cite[Theorem 5.3]{survey}\label{survey}
For any graph $G$, $\throt_{\Z}^*(G)$ is the least $k$ such that $\pt_{\Z}(G,k)=1$, i.e., $\throt_{\Z}^*(G)=k_{\Z}(G,1)$. Necessarily, $k_{\Z}(G,1)\geq \frac{n}{2}.$
\end{thm}

Theorem \ref{survey} states that the standard product throttling number for a graph is always achieved by a set that performs all of its forces in one time step. 
Since the condition that $\throt_{\Z}^*(G)>|V(G)|-k$ implies that no zero forcing set of $G$ forces $k$ or more vertices in one time step, Theorem \ref{survey} immediately gives that $\{G: \throt_{\Z}^*(G)> |V(G)|-k\}$ is characterized by a finite family of forbidden induced subgraphs. 
In particular, the forbidden family is given by \[\mathcal G=\{G:|V(G)|=2k,\exists B\subset V(G)\text{ s.t. } |B\up 1|=k, \pt_{\Z}(G;B)=1\}.\]
In \cite{survey}, this family is considered as $M$-sum graphs where $M$ is a $k$-matching.

Recall that in the context of standard and PSD zero forcing, that the terminus of a set of forces $\calf$ is the set of vertices in $V(G)$ that do not perform a force in $\calf$. 
Furthermore, the the reversal of a forcing set $B$, denoted $\operatorname{rev}(B)$, is the terminus of an arbitrary set of forces of $B$.
A key fact is that $\operatorname{rev}(B)$ is a (PSD) zero forcing set.

The proof of Theorem \ref{survey} relies on two facts about the standard zero forcing process. 
First, a reversal of a zero forcing set $B$ of size $|B|$ is also a zero forcing set of size $|B|.$ 
Second, $\pt(G;\text{rev}(B)\cup B)\leq \frac{\pt(G;B)}{2}.$
In combination, these two facts imply that for any zero forcing set $B$, we can find a zero forcing set that is at most twice as large as $B$, but propagates in at most half the time of $B$. 
In the context of product throttling, this is enough to conclude that any product throttling number can be realized by a zero forcing set that performs all its forces in one time step. 

Unfortunately, not every zero forcing rule has nice reversals (or a reversal may not be a $\X$ forcing set). 
In particular, if $\operatorname{rev}_+(B)$ is a reversal of a PSD zero forcing set $B$, then $|\operatorname{rev}_+(B)|$ can be much larger than $|B|$ (consider a large $d$-ary tree). 
Furthermore, $\pt_+(G;\operatorname{rev}_+(B)\cup B)$ only improves on $\pt_+(G;B)$ by a factor of around $1/2$. 

The goal of the remainder of this section is to prove versions of Theorem \ref{survey} for color change rules where the reversal based proof from \cite{survey} does not generalize. 
Subsection \ref{sec:PSDproduct} presents the PSD analog of Theorem \ref{survey}. 
While interesting in its own right, the PSD analog also provides a road map for the kinds of properties an abstract color change rule should have to make an alternative proof of Theorem \ref{survey} work. 
Subsection \ref{sec:genproduct} will define some properties of abstract color change rules that generalize Theorem \ref{survey} to a number of different zero forcing color change rules.

\subsection{PSD product throttling}\label{sec:PSDproduct}

Theorem \ref{psdprod} is the PSD zero forcing analog of Theorem \ref{survey}. 
Since we cannot generalize the reversal based proof of Theorem \ref{survey} to PSD zero forcing, we will prove Theorem \ref{psdprod} using a set of forbidden subgraphs.

\begin{thm}\label{psdprod}
For a positive integer $k$, let  $\mathcal G_{+,k}$ be the set of graphs $G$ with at most $2k$ vertices such that there exists an PSD forcing set $B$ and set of  forces $\calf$ with $\pt_+(G;B)=\pt_+(G;\calf)=1$ and $|\calf\up 1|=k.$

\begin{enumerate}
\item If $G$ contains a graph in $\mathcal G_{+,k}$, then $|V(G)|-k\geq \throt^*_{+}(G)$.

\item Suppose $G$ is a non-empty graph on $n$ vertices that does not contain a graph in $\mathcal G_{+,k}$ as an induced subgraph. 
If  $1<k<(n/12)^{1/3}$, then $\throt^*_{+}(G)=k_{+}(G,1)>|V(G)|-k.$ 

\end{enumerate}
\end{thm}

\begin{proof}[Proof of 1.] Suppose $G$ contains a graph in $\mathcal G_{+,k}$ as an induced subgraph on vertex set $S\subseteq V(G)$. 
Since $G[S]\in \mathcal G_{+,k}$, there exists an PSD forcing set $B_S$ such that $\pt_+(G[S];B_S)=1$ and $|S|-|B_S| = k$. 
Therefore, 
$\throt^*_{+}(G[S])\leq |S|-k$. 
Since the PSD zero forcing color change rule  is well behaved, this implies that $\throt^*_{+}(G)\leq |V(G)|-k.$
\end{proof}

\begin{proof}[Proof of 2.]
Suppose that $B$ is a PSD zero forcing set of $G$ with propagation time $t$ that achieves the PSD product throttling number of $G$. 
Let $|B|=b$. 
Let $\calf$ be a set of forces which realizes the propagation time of $B$ in $G$. 

Since $G$ does not contain a graph in $\mathcal G_{+,k}$, all forcing chains of $B$ given $\calf$ have at most $3k-2$ vertices. 
To prove this claim, suppose that 
\[x_1\fs x_2\fs\cdots\fs x_{3k-1}\]
is a PSD forcing chain of $B$ given $\calf$. 
Let $S=\{x_i: i \not\equiv 0 \mod 3\}$ and notice that $G[S]$ is a graph on $2k$ vertices. 
Furthermore, by the definition of PSD zero forcing and propagation time $x_i$ is adjacent to $x_j$ if and only if $j=i+1$ or $j=i-1$. 
In particular, $G[S]$ is a matching with $k$ edges. 
Therefore, $B_S=\{x_i: i\equiv 1\mod 3\}$ is a PSD zero forcing set of $G[S]$ with $\pt_+(G[S],B_S)=1$ and $|S|-|B_S|=k$. 
This contradicts the fact that $G$ does not contain a graph in $\mathcal G_{+,k}$.

Furthermore, since $G$ does not contain a graph in $\mathcal G_{+,k}$, each time step of forcing has at most $k-1$ forces. 
To prove this claim, suppose that $|\calf \up i|\geq k$ for some $1\leq i\leq t$. 
Let $F_1\subseteq \calf\up i$ such that $|F_1|= k$, and let $U_1 = \{x\in U\up i: x\fs y \in \calf, y \in F_1\}$.
Let $S= U_1\up i \cup \calf_1\up i= U_1\up i \cup F_1$
Since the PSD zero forcing color change rule is well behaved, we have that $|S|\leq 2k$, $\pt_+(G[S];U_1\up i)=1$ and $|F_1|=k$. 
Thus, $G[S]\in \mathcal G_{+,k}$, which is a contradiction.
Therefore, \[t(k-1)\geq n-b\quad \text{and}\quad b(3k-3)(k-1) \geq n-b.\]
In particular, the second bound comes from the fact that a PSD forcing tree cannot have a level of size $k$ or greater.
Putting the two bounds together gives
\[tb(3k-2)(k-1)^2\geq (n-b)^2.\]

For the sake of contradiction, suppose that $b\leq n/2$
Recall that $tb\leq n-1$, so we have 
\begin{align*}
    n(3k-3)(k-1)^2&\geq n^2/4.
\end{align*}
This is also a contradiction for large enough $n$. 
Therefore, $b\geq n/2$. 

Since $b\geq n/2$, it must follow that $t=1$. 
Finally, $\throt^*_{+}(G)=k_{+}(G,1)>|V(G)|-k$ because at most $k-1$ vertices can be forced in one time step.

\end{proof}

The key insight in the proof of Theorem \ref{psdprod} is that long PSD zero forcing chains will induce graphs in the forbidden family. 
The proof of this fact used specific knowledge about the PSD zero forcing color change rule. 
In particular, we know that PSD forces must occur on edges, and therefore, PSD forcing chains that realize the propagation time induce paths.

It does not seem like using long forcing changes to find forbidden graphs will work for any well behaved color change rules.
To see this, suppose that 
\[x_1\fs x_2\fs\cdots\fs x_{3k-1}\]
is an $\X$ forcing chain for a well behaved color change rule $\X$, and define $S$ and $B_S$ as in the proof of Theorem \ref{psdprod}.
We can use the fact that $\X$ is well behaved to conclude that $B_S$ is an $\X$ forcing set of $G[S]$. 
However, it is not clear why $B_S$ will turn $G[S]$ blue in one time step without more information about $\X$.
In particular, even if $\X$ forces only occur on edges (which might not be true), there is still no guarantee that the forcing chain induces a path, and that $G[S]$ is a matching. 
These difficulties will be tackled in the next subsection. 

\subsection{Product throttling for color change rules}\label{sec:genproduct}
The next theorem generalizes some of the implications of Theorem \ref{survey} to abstract color change rules at the cost of some strength. 
In particular, high $\X$ product throttling numbers for sufficiently large graphs are characterized by a finite family of forbidden subgraphs (and this family is analogous to the family $\mathcal G$ above). 
Additionally, these high $\X$ product throttling numbers are realized by $\X$ forcing sets that perform all their forces in one time step. 
The costs of the theorem are the restriction to high product throttling numbers where we already know that the initial blue set must have at least $k\sqrt{3n}$.

An $\X$ force $x\fs y$ is \emph{independent} of another vertex $v$ if knowing the color of $v$ is not required for determining whether $x$ can force $y$. 
For example, in standard zero forcing, $x\fs y$ is independent of all $v\notin N[x]$. 
We say a forcing chain $x_0\fs x_1\fs \cdots\fs x_r$ is \emph{internally independent} if we have $x_j\fs x_{j+1}$ is  independent of  $x_i$ for all $i\neq j-1,j,j+1$.

An internally independent chain will let us identify forces that can happen simultaneously for an appropriate set of blue vertices.
Therefore, if there exists a large internally independent chain, then it should be relatively easy to find forcing sets with controlled throttling behavior.

\begin{defn}\label{defn:lss}
Let $\X$ be a color change rule.
\begin{itemize}
\item We say $\X$ is \emph{local} if for all $G$ and any $v,w\in V(G)$, we have that $v\overset{\X}{\fs} w$ is independent of $V(G)\setminus N[v]$.

\item We say a  color change rule $\X$ is  \emph{symmetric} if $v\overset{\X}{\fs} w$ is valid given blue set $B$ for some $w\in (V(G)\setminus B)\cap N[v]$ implies  that $v\overset{\X}{\fs}  w'$ is valid given $B$ for all $w'\in (V(G)\setminus B)\cap N[v]$.

\item We say a color change rule $\X$ is \emph{simple} if whenever $u\overset{\X}{\fs}v$ and $x\overset{\X}{\fs} y$ are valid given $B$ (with $v\neq y$), then $u\overset{\X}{\fs}v$ and $x\overset{\X}{\fs} y$ can be performed simultaneously. 

\item We say that a color change rule $\X$ is an \emph{infection rule} (or an infectious color change rule), if $u\overset{\X}{\fs} v$ is valid at time step $t$ implies that $u$ is blue at time $t-1$.
\end{itemize}
\end{defn}

\begin{obs}
Suppose $\X$ is a color change rule.
If $\X$ is local, then $v\fs w$ for non-adjacent $v$ and $w$ is not valid (and impossible). 
Furthermore, if $v\fs w$ is possible for non-adjacent $v$ and $w$, then $\X$ is not local. 
\end{obs}

Essentially, a color change rule $\X$ is local and symmetric  when a vertex $v$ can force a white neighbor if and only if $v$ can force any of its white neighbors.
Standard zero forcing, skew forcing, and  $k$-forcing are local, symmetric, and simple color change rules.
On the other hand, PSD zero forcing is not even local. 
To see this, let $C_4$ have vertices labeled $1,2,3,4$ clockwise with $B=\{1,3\}$ and $B'=\{1\}$. 
In this case, $1\overset{+}\fs 2$ is valid for blue set $B$, but not $B'$. 
In particular, the validity of $1\overset{+}\fs 2$ depends on the color of vertex $3$ which is not a neighbor of $1$. 

This highlights the shortcomings of the PSD color change rule. 
The PSD color change rule requires that white vertices which are simultaneously forced by the same blue vertex $v$ are not in the same component (which is not a property determined by the neighborhood of $v$).

Any example we have of a color change rule that is local but not symmetric is contrived, and does not arise from a natural application (unlike most zero forcing rules). 
For example, we could insist that the vertices of $G$ are ordered, and that a blue vertex $v$ can only force its smallest white neighbor. 
Furthermore, color change rules defined on ordered graphs is beyond the purview of this discussion.

Finally, we want to point out that simple and symmetric do not imply each other. 
To see this consider zero forcing with hopping (also known as the minor monotone floor of zero forcing $\zf$).
In zero forcing with hopping, a blue vertex $v$ without white neighbors such that $v$ has not performed a force may force any white vertex in the graph. 
The zero forcing with hopping color change rule is vacuously symmetric since it follows the standard zero forcing color change rule when vertices are adjacent. 
However, it is possible that an isolated blue vertex has 2 white non-neighbors which cannot be forced simultaneously. 
For work on the throttling of the minor monotone floor of the standard zero forcing rule, see \cite{JCThrot}. For throttling where hopping is the only allowed color change rule, see \cite{CP22}.

Notice that the minor monotone floor of standard zero forcing is not a simple color change rule. 
In particular, if $G=\overline{K_3}$ with vertex set $\{1,2,3\}$ and blue set $B=\{1\}$, then $1\fs 2$ and $1\fs 3$ are both valid in the first time step. 
However, these two forces cannot be performed simultaneously. 

We can always find a set of forces which minimizes the time step at which a particular vertex turns blue. 
However,
under certain conditions, we can find a single set of forces which minimizes the time step at which each vertex turns blue. 
To this end, consider the following definition.
We say a set of forces $\calf_*$ is \emph{uniformly as fast as possible} for a blue set $B$ if for all vertices $v\in V(G)\setminus B$ and all sets of forces $\calf$ we have that $v\in \calf_*\up i\cap \calf\up j$ implies that $i\leq j$.

\begin{lem}\label{lem:uniform}
Let $\X$ be simple and nearly well behaved.  If $B$ is an $\X$ forcing set, then there exists a set of forces which is uniformly as fast as possible.
\end{lem}

\begin{proof}
For a set of forces $\calf$, let $m(\calf)$ be the number of vertices $v$ for which $\calf$ minimizes the time step at which $v$ turns blue. 
Let $\calf$ be the set of forces for which $m(\calf)$ is maximized. 
For the sake of contradiction, suppose $m(\calf)<n-|B|$. 
This implies that there exists vertex $v$ and set of forces $\calf_v$ such that $\calf\up i\cap \calf_v\up j$ with $i>j$ which minimizes $j$. 
In particular, $j$ is the first time step when $\calf_v\upc j\not \subseteq \calf\upc j$. 

Let $u$ be the vertex such that $u\fs v\in \calf_v$.  
Furthermore, let $x$ be the vertex such that $x\fs v\in \calf.$
Since $\calf_v\upc k \subseteq \calf\upc k$ for all $k < j$ and $\X$ is well behaved, we have that $u \fs v$ is valid at time $k$ given $\calf$.
We claim that $\calf_*=\calf\cup\{u\fs v\}\setminus \{x\fs v\}$ is a set of forces with $m(\calf_*)> m(\calf).$

First we will show that $\calf_*$ is a set of forces by constructing a chronological list of forces from which $\calf_*$ can be derived. 
First, greedily order forces in $\calf\cap \calf_*$ until $\calf\up {j-1}$ is blue. 
Next, append $u\fs v$. 
Now, greedily add forces in $\calf\cap \calf_*\setminus\{ x\fs v\}$ until $\calf\up i$ is blue. 
Since $v$ is already blue, the force $x\fs v$ is not necessary.
Finally, use the remaining forces to finish coloring $G$ blue.\footnote{A similar method of transitioning from one set of forces to another is used in \cite{AKWY22}, but proven in less generality and with significantly different notation.} 

Next we will show that $m(\calf_*)>m(\calf).$
Clearly, $\calf_*\up k=\calf\up k$ for $k<j$. 
Since $\X$ is simple, $\calf_*\up j=\calf \up j\cup \{v\}$.
Furthermore, $\calf_*\up k\subseteq \calf\up k $ since $\X$ is simple for $k> j$. 
Therefore, $m(\calf_*)>m(\calf).$
\end{proof}

The next lemma takes a set of forces that is uniformly as fast as possible, and shows that its forcing chains are internally independent and induce paths. 

\begin{lem}\label{lem:internallyindependent}
Let $\X$ be a well behaved, local, symmetric, and simple color change rule.
If $B$ is an $\X$ forcing set of $G$, then there exists $\calf$ such that every forcing chain in $\calf$ is an induced path in $G$. 
Furthermore, every chain in $\calf$ is internally independent. 
\end{lem}

Note that $\calf$ can also be taken to be uniformly as fast as possible.

\begin{proof}
Let $G$ be a graph with $\X$ forcing set $B$. 
By Lemma \ref{lem:uniform} there exists a set of forces $\calf$ which is uniformly as fast as possible.  
Let $x_0\fs x_1\fs \cdots \fs x_r$ be a $\X$ forcing chain of $B$ in $\calf$. 
Without loss of generality, there is a time $c$ such that $x_0\fs x_1$ is valid when the set of blue vertices in $G$ is  $\calf\upc c$.
We will show by induction on $j$ that $x_j\fs x_{j+1}$ is independent from the colors of $x_i$ for $i \neq j-1,j,j+1$.

For the sake of contradiction, suppose that $x_0$ and  $x_i$ are adjacent for some $i\geq 2$. 
Since $x_0\fs x_1$ is valid given blue set $\calf\upc c$ and $\X$ is symmetric, it follows that $x_0\fs x_i$ is valid at time $c$. 
This contradicts the fact that $\calf$ is uniformly as fast as possible. 
Therefore, $x_0$ is not adjacent to $x_i$ for any $i\geq 2$.
Furthermore, since $\X$ is local, the validity of $x_0\fs x_1$ does not depend on the non-neighbors of $x_0$. 
In particular, $x_0\fs x_1$ does not depend on the color of $x_i$ for $i\geq 2$.

As the strong induction hypothesis, assume that for all $k<j$ we have $x_k\fs x_{k+1}$ is independent of $x_i$ for all $i\neq k-1,k,k+1$ and $x_k$ is not adjacent to $x_i$ for all $i \neq k-1,k,k+1$. 

We now prove the claim for index $j$. 
By the induction hypothesis, $x_j$ is not adjacent to $x_i$ for all $i <j-1$. 
For the sake of contradiction, suppose that $x_j$ is adjacent to $x_i$ for some $i\geq j+2$. 
Since $x_j\fs x_{j+1}$ is valid when $\calf^{[j]}$ is the blue set and $\X$ is symmetric, it follows that  $x_j\fs x_i$ is valid when $\calf^{[j]}$ is the blue set of $G$. 
This contradicts the fact that $\calf$ is uniformly as fast as possible.

Notice that the only neighbors of $x_i$ are $x_{i-1}$ and $x_{i+1}$ for $0\leq i \leq r$ (where the index makes sense). 
This implies that $x_0\neq x_r$ if $r\geq 3$. 
Finally, if $x_2 =x_0$, then the chain still induces a path (this needs to be considered for skew forcing).
\end{proof}

Recall that in Subsection \ref{sec:PSDproduct}, we were able to show that forcing chains of a PSD zero forcing set induce paths and are internally independent. 
The work done in Lemma \ref{lem:internallyindependent} draws the same conclusions for abstract local, symmetric, and simple color change rules. 
Interestingly, the PSD color change rule is not local, suggesting that the property of being local as defined in Definition \ref{defn:lss} is too narrow. 
Regardless, Definition \ref{defn:lss} is still worth considering since it captures standard zero forcing, skew forcing, and $k$-forcing for all $k\geq 1$.

\begin{thm}\label{genprod}
Suppose $\X$ is a well behaved, local, symmetric, simple, and infectious color change rule. 
For a positive integer $k$, let $\mathcal G_{\X,k}$ be the set of graphs $G$ with at most $2k$ vertices such that there exists an $\X$ forcing set $B$ and set of  forces $\calf$ with $\ptx(G;B)=\ptx(G;\calf)=1$ and $|\calf\up 1|=k.$

\begin{enumerate}
\item If $G$ contains a graph in $\mathcal G_{\X,k}$, then $|V(G)|-k\geq \throt^*_{\X}(G)$.

\item Suppose $G$ is a non-empty graph on $n$ vertices that does not contain a graph in $\mathcal G_{\X,k}$ as an induced subgraph. 
If $1<k<(n/12)^{1/3}$, then $\throt^*_{\X}(G)=k_{\X}(G,1)>|V(G)|-k.$ 

\end{enumerate}
\end{thm}

\begin{proof}[Proof of 1.] Suppose $G$ contains a graph in $\mathcal G_{\X,k}$ as an induced subgraph on vertex set $S\subseteq V(G)$. 
Since $G[S]\in \mathcal G_{\X,k}$, there exists an $\X$ forcing set $B_S$ such that $\ptx(G[S];B_S)=1$ and $|S|-|B_S| = k$. 
Therefore, 
$\throt^*_{\X}(G[S])\leq |S|-k$. 
Since $\X$ is well behaved, this implies that $\throt^*_{\X}(G)\leq |V(G)|-k.$
\end{proof}

\begin{proof}[Proof of 2.] Suppose that $B$ is an $\X$ forcing set of $G$ with propagation time $t$ that achieves the $\X$ product throttling number of $G$.
Let $|B|=b$.
By Lemma \ref{lem:internallyindependent}, there exists a set of forces $\calf$ which realizes the propagation time, is uniformly as fast as possible, and has internally independent chains.

Since $G$ does not contain a graph in $\mathcal G_{\X,k}$, all  forcing chains of $B$ given $\calf$ have at most $3k-2$ vertices. 
To prove this claim, suppose that 
\[x_1\fs x_2\fs\cdots\fs x_{3k-1}\] is a $\X$ forcing chain of $B$ given $\calf$.
Let $S=\{x_i: i\not\equiv 0 \mod 3\}$ and notice that $G[S]$ is a graph on $2k$ vertices such that $B_S=\{x_i: i\equiv 1\mod 3\}$ is an $\X$ forcing set with $\ptx(G[S],B_S)=1$ and $|S|-|B_S|=k$. 
In particular, $\ptx(G[S],B_S)=1$ since chains in $\X$ are internally independent. 
This contradicts the fact that $G$ does not contain a graph in $\mathcal G_{\X,k}.$

Furthermore, since $G$ does not contain a graph in $\mathcal G_{\X,k}$, each time step of forcing has at most $k-1$ forces.

Therefore, \[t(k-1)\geq n-b\quad \text{and}\quad b(3k-3)(k-1) \geq n-b.\]
Putting the two bounds together gives
\[tb(3k-2)(k-1)^2\geq (n-b)^2.\]

For the sake of contradiction, suppose that $b\leq n/2$
Recall that $tb\leq n-1$, so we have 
\begin{align*}
    n(3k-3)(k-1)^2&\geq n^2/4.
\end{align*}
This is also a contradiction for large enough $n$. 
Therefore, $b\geq n/2$. 

Since $b\geq n/2$, it must follow that $t=1$. 
Finally, $\throt^*_{+}(G)=k_{+}(G,1)>|V(G)|-k$ because at most $k-1$ vertices can be forced in one time step.
\end{proof}

\section{Co-local and co-symmetric color change rules}\label{sec:colocal}

In \emph{$r$-bootstrap percolation}, we start with a graph $G$ and an initial set of infected vertices $B = B_0 \subseteq V(G)$. For each $t = 1, 2, 3,\ldots$, we let
\[
B_t = B_{t-1} \cup \{v \in V(G) : |N(v) \cap B_{t-1}| \geq r\}
\]
be the set of vertices that become infected at time step $t$.
If there exists a $k \geq 0$ such that $\bigcup_{t=0}^k B_t = V(G)$, we say that $B$ is a \emph{$r$-percolating set} of $G$ (or $B$ \emph{$r$-percolates} $G$). 
The smallest such $k$ is the \emph{$r$-percolation time} of $B$ in $G$, and is denoted by $\ptime$. 
We define the \emph{throttling number for $r$-bootstrap percolation} of $G$ as
\[
\thperc(G) := \text{min}\{|B| + \ptime : B \text{ $r$-percolates } G \}.\]

Notice that $r$-bootstrap percolation can be phrased as a color change rule, where a blue vertex $v$ forces a white vertex $w$ to turn blue if $v\in N(w)$ and $|N(w)\cap B|\geq r.$
With this description, $r$-bootstrap percolation is a simple and nearly well behaved color change rule. 
However, unlike zero forcing color change rules, the validity $u\fs v$ for the $r$-bootstrap percolation color change rule is determined by the neighborhood of $v$ and not the neighborhood of $u$.
This motivates the following definitions.

\begin{defn}\label{defn:cols}
Let $\X$ be a color change rule.
\begin{itemize}
\item We say $\X$ is \emph{co-local} if for all $G$ and any $v,w\in V(G)$, we have that $v\overset{\X}{\fs} w$ is independent of $V(G)\setminus N[w]$.

\item We say a color change rule $\X$ is  \emph{co-symmetric} if $v\overset{\X}{\fs} w$ is valid given blue set $B$ for some $w\in B\cap N[w]$ implies  that $v'\overset{\X}{\fs}  w$ is valid given $B$ for all $v'\in B\cap N[w]$.

\end{itemize}
\end{defn}

Let $\ell(\mathcal F)$ denote the length of the longest forcing chain in $\mathcal F$.
In general $\ell(\mathcal F)\leq \pt_{\Z}(G;\calf)$, even for nearly well behaved, local, symmetric, and simple color change rules. 
Figure \ref{fig:chainVSpt} shows that $\ell(\calf)< \pt_{\Z}(G,B)$ for some $G$ and $B$ with the standard zero forcing color change rule. 

\begin{prop}\label{prop:maxchain}
Let $\X$ be a nearly well behaved, co-local,  co-symmetric, and simple infection rule.
If $B$ is an $\X$ forcing set of $G$, then 
\[\pt_{\X}(G;B) = \max\{\ell(\mathcal F): \mathcal F\text{ is a uniformly as fast as possible set of forces of $B$}\}.\]
\end{prop}

\begin{proof}
Assume that $B$ is an $\X$ forcing set of $G$.
Since $\X$ is nearly well behaved, there exists a set of $\X$ forces that is uniformly as fast as possible.
We claim that we choose a uniformly as fast as possible $\mathcal F_*$ such that if $u\fs v\in \mathcal F_*$ and $v\in \mathcal F_*\up i$, then $u\in \mathcal F_*\up{i-1}$ by the fact that $\X$ is infectious, co-symmetric, and simple. 
That is, whenever there are multiple vertices that can force $v$, then assign responsibility for the force to the most recently forced vertex.
To prove this claim, suppose that $\calf_*$ is uniformly as fast as possible such that $\calf_*$ maximizes the number of forces $x\fs y$ where  $y\in \calf_*\up i$ and $x\in \calf_*\up {i-1}$ for some $i$.
For the sake of contradiction, assume that $u\fs v\in \calf_*$ where $v\in \calf_*\up i$ and $u\in \calf_* \up j$ with $j<i-1$ (since $\X$ is infectious, $u$ must be blue at time step $i$).
If $v$ does not have a neighbor in $\calf_*\up {i-1}$, then $u\fs v$ is valid at time step $i-1$ since $\X$ is co-local and simple.
Therefore, $v$ must have a neighbor in $u'\in \calf_*\up {i-1}$. 
Since $\X$ is co-symmetric and simple, $u'\fs v$ is valid at time step $i$. 
Thus, $(\calf_* \cup \{u'\fs v\} )\setminus \{u\fs v\}$ is a uniformly as fast as possible set of forces for $B$ with more forces $x\fs y$ where  $y\in \calf_*\up i$ and $x\in \calf_*\up {i-1}$ for some $i$.
This contradicts the assumption that $\calf_*$ maximizes the number of forces with the desired property, proving the claim.

If $\mathcal F$ is uniformly as fast as possible, $t=\pt_{\X}(G;B)=\pt_{\X}(G;\mathcal F).$
Furthermore, $\ell(\calf)\leq \pt_{\X}(G;\calf)$ for all set of forces $\calf$.
Since $\X$ is infectious, $x_0\in B$ for all maximal forcing chains \[x_0\fs x_1\fs \cdots\fs x_k\] in $\calf_*$.
If we further assume that the forcing chain is maximum, then $k=t$ since $\calf_*\up t$ is non-empty.
Thus, $\ell(\calf_*)=t=\pt_{\X}(G;\mathcal F)=\pt_{\X}(G;B)$, completing the proof.
\end{proof}
\begin{figure}
\begin{center}
\scalebox{2}{
\begin{tikzpicture}
\draw (0,0) coordinate (c1) node[vtx, blue](x1){}
(1,0) coordinate (c2) node[vtx, fill=white](x2){}
(2,0) coordinate (c3) node[vtx, fill=white](x3){}
(0,1) coordinate (c4) node[vtx, blue](x4){}
(1,1) coordinate (c5) node[vtx, fill=white](x5){}
(2,1) coordinate (c6) node[vtx, fill=white](x6){};

\draw (x1) to (x2) to (x3);
\draw (x4) to (x5) to (x6); 
\draw (x1) to (x5) to (x3);
\draw (x1) to (x4)
(x2) to (x5)
(x3) to (x6);
\end{tikzpicture}
}
\end{center}
\caption{ A graph $G$ with initial blue set $B$. 
Each initial blue vertex in the depicted graph must perform a force.  Furthermore, $\pt_{\Z} (G;B) = 4$.
This shows that any uniformly as fast as possible set of forces $\calf$ for $B$ has $\ell(\calf)< \pt_{\Z}(G,B).$\label{fig:chainVSpt}}
\end{figure}
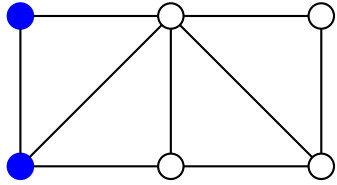

\section{Concluding remarks}

A general idea within zero forcing research is to add constraints to parameters so that the parameter becomes minor monotone. 
An example of this idea is the minor monotone floor of the zero forcing number, denoted $\lfloor \Z\rfloor$ (see \cite{Parameters, JCThrot}). 
The motivation for finding minor monotone analogs of zero forcing parameters is that they allow an application of the Graph Minor Theorem. 
In short, the Graph Minor Theorem states that the set of all graphs ordered by minor containment is a well-quasi-ordering, and in particular, any infinite family $F$ of graphs upwardly closed under the minor relation is characterized by a finite set of minor minimal graphs. 
For example, the set of all graphs containing a cycle $\mathcal C$ is an infinite family of graphs which is characterized by the single minor minimal graph $K_3$ (any cycle contains $K_3$ as a minor). 
An application of the Graph Minor Theorem is useful as it immediately motivates the search for sets of minor minimal graphs that characterize an infinite set.

In contrast, the set of graphs ordered by induced subgraph inclusion does not lend itself nicely to characterizations by finite families of minimal  graphs.
Here again, $\mathcal C$ is instructive because there is no finite family which characterizes $\mathcal C$  under induced subgraph inclusion, since each cycle length must be forbidden individually. 
In particular, the set of cycles is an infinite anti-chain in the induced subgraph partial order on graphs. 

The result of our two theorems shows that some infinite families of graphs characterized by a particular throttling behavior are also characterized by finite families under induced subgraph inclusion. 
This fact is remarkable in part because the there is no \emph{a priori} guarantee that infinite families should have a characterization by a finite family of minimal graphs.


\end{document}